\documentclass{amsart}
\usepackage{geometry}
\newtheorem{theorem}{Theorem}[section]
\newtheorem{lemma}[theorem]{Lemma}
\newtheorem{proposition}[theorem]{Proposition}
\newtheorem{corollary}[theorem]{Corollary}
\theoremstyle{definition}
\newtheorem{definition}[theorem]{Definition}

\newcommand{\ks}{{\mathcal K\,}}
\newcommand{\hs}{{\mathcal H\,}}
\newcommand{\C}{{\mathcal C\,}}
\newcommand{\z}{\mathbb Z\,}
\theoremstyle{remark}

\numberwithin{equation}{section}
\allowdisplaybreaks[4]

\begin{document}

\title{Generalized frames and controlled  operators in Hilbert space}

\author{Dongwei Li}
\address{School of Mathematical Sciences, University of Electronic Science and Technology of China, 611731, P. R. China}
\email{ldwmath@163.com}

\author{Jinsong Leng}
\address{School of Mathematical Sciences, University of Electronic Science and Technology of China, 611731, P. R. China}
\email{jinsongleng@126.com}




\begin{abstract}
Controlled frames and g-frames were considered recently as generalizations of frames in Hilbert spaces. In this paper we generalize some of the known results in frame theory to controlled g-frames. We obtain some new properties of controlled g-frames and obtain new controlled g-frames by considering controlled g-frames for its components. And we obtain some new resolutions of the identity. Furthermore, we study the stabilities of controlled g-frames under small perturbations.
\end{abstract}

\maketitle

\section{Introduction}
Frames were first introduced in 1952 by Duffin and Schaeffer \cite{duffin1952class} to study some problems in nonharmonic Fourier series, and were widely studied from 1986 since the great work by Daubechies et al. \cite{Daubechies1986}. To date, frame theory has broad applications in pure mathematics, for instance, Kadison-Singer problem  and statistics, as well as in applied mathematics, computer science and engineering applications. We refer to \cite{bodmann2005frames,kovavcevic2008introduction,leng2011optimal,leng2011optimald,morillas2017optimal,Sun2015Nonlinear} for an introduction to frame theory and its applications.   

In 2006, Sun \cite{sun2006g} introduced the concept of g-frame. G-frames are generalized frames, which include ordinary frames, bounded invertible linear operators, fusion frames, as well as many recent generalizations of frames. For more details see \cite{guo2014characterizations,sun2007stability}.
G-frames and g-Riesz bases in Hilbert spaces have some properties similar to those of frames, but not all the properties are similar \cite{sun2006g}. 
Controlled frames for spherical wavelets were introduced in \cite{bogdanova2005stereographic} and were reintroduced recently to improve the numerical efficiency of iterative  algorithms \cite{balazs2010weighted}. The role of controller operators is like the role of preconditions matrices or operators in linear algebra. 

Controlled g-frames were introduce by Khosravi et al. \cite{rahimi2013controlled}.  Now, many excellent results of controlled frames have been achieved and applied successfully \cite{balazs2010weighted,balazs2012multipliers}, which properties of the controlled frames may be extended to the g-frames? It is a tempting subject because of the complexity of the structure of g-frames compared with conventional frames. In this paper, we give some new properties of controlled g-frames and construct new controlled g-frames from a given controlled g-frame, and we generalize some of known results in g-frames to controlled g-frames in Section 2. In section 3 we obtain some new  resolutions of the identity with controlled g-frames, and in Section 4 we study the stability of controlled g-frames under small perturbations.

Throughout this paper, $\hs$ and $\ks$ are two separable Hilbert spaces and $\{\hs_i:i\in I\}$ is a sequence of subspaces of $\ks$, where $ I$ is a subset of $\z$. $L(\hs,\hs_i)$ is the collection of all bounded linear operators from $\hs$ into $\hs_i$, and $GL(\hs)$ denotes the set of all bounded linear operators which have bounded inverse. It is easy to see that if $T,U\in GL(\hs)$, then $T^*,T$ and $TU$ are also in $GL(\hs)$. Let $GL^{+}(\hs)$ be the set of positive operators in $GL(\hs)$. Also $I_{\hs}$ denotes the identity operator on $\hs$.

Note that for any sequence $\{\hs_i:i\in I\}$ of Hilbert spaces, we can always find a large Hilbert space $\ks$ such that for all $i\in I,~\hs_i\subset\ks$ (for example $\ks=\oplus_{i\in I}\hs_i$).
\begin{definition}
	A sequence $\Lambda=\{\Lambda_i\in L(\hs,\hs_i):i\in I\}$ is called a generalized frame, or simply a g-frame, for $\hs$ with respect to $\{\hs_i:i\in I\}$ if there exist constants $0<A\le B<\infty$ such that 
	$$A\|f\|^2\le \sum_{i\in  I}\left\|\Lambda_if \right\| ^2\le B\left\| f\right\|^2,~~~{\rm{for~ all }}~ f\in\hs.\eqno(1)$$
	The numbers $A$ and $B$ are called g-frame bounds.
\end{definition}
We call $\Lambda$ a tight g-frame if $A=B$ and Parseval g-frame if $A=B=1$. If the second inequality in (1) holds, the sequence is called g-Bessel sequence.

$\Lambda=\{\Lambda_i\in L(\hs,\hs_i):i\in I\}$ is called a g-frame sequence, if it is a g-frame for ${\rm{\overline{span}}}\{\Lambda^*_i(\hs)\}_{i\in I}$.

For each sequence $\{\hs_i\}_{i\in I}$, we define the space $(\sum_{i\in I}\oplus\hs_i)_{\ell_2}$ by
$$(\sum_{i\in I}\oplus\hs_i)_{\ell_2}=\big\{\{f_i\}_{i\in I}:f_i\in\hs_i,i\in I~{\rm and}~\sum_{i\in I}\|f_i\|^2<+\infty\big\}$$
with the inner product defined by 
$$\left\langle \{f_i\},\{g_i\}\right\rangle =\sum_{i\in I}\left\langle f_i,g_i\right\rangle .$$
\begin{definition}
	Let $\Lambda=\{\Lambda_i\in L(\hs,\hs_i):i\in I\}$ be a g-frame for $\hs$. Then the synthesis operator for $\Lambda=\{\Lambda_i\in L(\hs,\hs_i):i\in I\}$ is the operator 
	$$\Theta_{\Lambda}:(\sum_{i\in I}\oplus\hs_i)_{\ell_2}\longrightarrow\hs$$
	defined by
	$$\Theta_{\Lambda}(\{f_i\}_{i\in I})=\sum_{i\in I}\Lambda^*_i(f_i).$$
\end{definition}
The adjoint $\Theta^*_{\Lambda}$ of the synthesis operator is called analysis operator which is given by

$$\Theta^*_{\Lambda}:\hs\longrightarrow (\sum_{i\in I}\oplus\hs_i)_{\ell_2},~~~T^*(f)=\{\Lambda_if\}_{i\in I}.$$

By composing $\Theta_{\Lambda}$ and $\Theta^*_{\Lambda}$, we obtain the g-frame operator 
$$S_{\Lambda}:\hs\longrightarrow\hs,~~S_{\Lambda}f=\Theta_{\Lambda}\Theta^*_{\Lambda}f=\sum_{i\in I}\Lambda^*_i\Lambda_if.$$
It is easy to see that g-frame operator is a bounded, positive and invertible operator.
\section{Controlled g-frames and constructing new controlled g-frames}
Controlled g-frames with two controlled operators were studied in \cite{musazadeh2016some, rahimi2013controlled}. Next, we give the definition of controlled g-frames.
\begin{definition}
	Let $T,U\in GL(\hs)$. The family $\Lambda=\{\Lambda_i\in L(\hs,\hs_i):i\in I\}$ is called a $(T,U)$-controlled g-frame for $\hs$, if $\Lambda$ is a g-Bessel sequence and there exist constants $0<A\le B<\infty$ such that 
	$$A\|f\|^2\le \sum_{i\in I}\left\langle \Lambda_iTf,\Lambda_iUf\right\rangle \le B\|f\|^2,~~~\forall f\in\hs.\eqno(2)$$
	$A$ and $B$ are called the lower and upper controlled frame bounds, respectively. 
\end{definition}
If $U=I_{\hs}$, we call $\Lambda=\{\Lambda_i\}$ a $T$-controlled g-frame for $\hs$ with bounds $A$ and $B$. If the second part of the above inequality holds, it is called $(T,U)$-controlled g-Bessel sequence with bound $B$.

Let $\Lambda=\{\Lambda_i\in L(\hs,\hs_i):i\in I\}$ be  a $(T,U)$-controlled g-frame for $\hs$, then the $(T,U)$-controlled g-frame operator is defined by

$$S_{T\Lambda U}:\hs\longrightarrow\hs,~~S_{T\Lambda U}f=\sum_{i\in I}U^*\Lambda^*_i\Lambda_iTf,~~\forall f\in\hs.$$
It follows from the definition that for a g-frame, this operator is positive and invertible and 
$$AI_{\hs}\le S_{T\Lambda U}\le BI_{\hs},$$
also $S_{T\Lambda U}=U^*S_{\Lambda}T$.

For convenience we state the following lemma.
\begin{lemma}\cite{balazs2010weighted}
	Let $T:\hs\longrightarrow\hs$ be a linear operator. Then the following conditions are equivalent:
	\begin{enumerate}
		\item There exist $m>0$ and $M<\infty$, such that $mI_{\hs}\le T\le MI_{\hs}$.
		\item $T$ is positive and there exist $m>0$ and $M<\infty$, such that 
		$$m\|f\|^2\le \|T^{1/2}f\|^2\le M\|f\|^2.$$
		\item $T\in GL^+(\hs).$
	\end{enumerate}
\end{lemma}
\begin{proposition}
	Let $T,U\in GL^+(\hs)$ and $\Lambda=\{\Lambda_i\in L(\hs,\hs_i):i\in I\}$ be a family of operator. Then the following statements hold.
	\begin{enumerate}
		\item  If $\{\Lambda_i:i\in I\}$ is a $(T,U)$-controlled g-frame for $\hs$, then $\{\Lambda_i:i\in I\}$ is a g-frame for $\hs$.
		\item If $\{\Lambda_i:i\in I\}$ is a g-frame for $\hs$ and $T,U\in GL^+(\hs)$, which commute with each other and commute with $S_{\Lambda}$, then $\{\Lambda_i:i\in I\}$ is a $(T,U)$-controlled g-frame for $\hs$.
	\end{enumerate}
\end{proposition}
\begin{proof}
	1. For $f\in\hs$, since the operator
	$$S_{\Lambda}(f)=(U^*)^{-1}S_{T\Lambda U}T^{-1}(f)=\sum_{i\in I}\Lambda^*_i\Lambda_if$$
	is well defined, we show that it is a bounded and invertible operator.  It is also a positive linear operator on $\hs$ because
	$$\left\langle S_{\Lambda}f,f\right\rangle =\sum_{i\in I}\|\Lambda_if\|^2.$$
	Hence, 
	$$\|S^{-1}_{\Lambda}\|=\|TS^{-1}_{T\Lambda U}U^*\|\le\|T\|\|S^{-1}_{T\Lambda U}\|\|U^*\|\le \frac{1}{A}\|T\|\|U^*\|,$$ which $A$ is the lower frame bound of  $(T,U)$-controlled g-frame $\{\Lambda_i:i\in I\}$.
	So $S_{\Lambda}\in GL^+(\hs)$. Therefore, by Lemma 2.2, we have $CI_{\hs}\le S_{\Lambda}\le DI_{\hs}$ for some $0<C\le D<\infty$. So the result holds.
	
	2. Let $\{\Lambda_i:i\in I\}$ be a g-frame with bounds $C,D$ and $m,m'>0$, $M,M'<\infty$ so that
	$$mI_{\hs}\le T\le MI_{\hs},~~~m'I_{\hs}\le U^*\le M'I_{\hs}.$$
	By Lemma 2.2, then we have 
	$$mCI_{\hs}\le S_{\Lambda}T\le MDI_{\hs}$$
	because $T$ commutes with $S_{\Lambda}$. Again $U^*$ commutes with $S_{\Lambda}T$ and then 
	$$mm'CI_{\hs}\le S_{T\Lambda U}\le MM'DI_{\hs.}$$
	This completes the proof.  
\end{proof}
\begin{corollary}
	Let $T,U\in GL(\hs)$ and $\{\Lambda_i:i\in I\}$ be a g-frame with frame operator $S_{\Lambda}$. If $U^*S_{\Lambda}T$ is positive, then $\{\Lambda_i:i\in I\}$ is a $(T,U)$-controlled g-frame for $\hs$.
\end{corollary}
The Theorem 2.8 of \cite{asgari2005frames} leads to the following result.
\begin{proposition}
	Let $T,U\in GL(\hs)$ and $\{\Lambda_i:i\in I\}$ be a $(T,U)$-controlled g-frame for $\hs$ with lower and upper bounds $A$ and $B$, respectively. Let $\{\Gamma_i:i\in I\}$ be a g-complete family of bounded operator. If there exists a number $0<R<A$ such that 
	$$0\le \sum_{i\in I}\left\langle U^*(\Lambda^*_i\Lambda_i-\Gamma^*_i\Gamma_i)Tf,f\right\rangle \le R\|f\|^2,~~\forall f\in\hs,$$
	then $\{\Gamma_i:i\in I\}$ is also a $(T,U)$-controlled g-frame for $\hs$.
\end{proposition}

\begin{proof}
	Let $\{\Lambda_i:i\in I\}$ be a $(T,U)$-controlled g-frame for $\hs$ with frame bounds $A$ and $B$, for any $f\in\hs$, we have
	$$A\|f\|^2\le \sum_{i\in I}\left\langle U^*\Lambda^*_i\Lambda_iTf,f\right\rangle \le B\|f\|^2.$$	
	Hence
	\begin{align*}
		\sum_{i\in I}\left\langle U^*\Gamma^*_i\Gamma_iTf,f\right\rangle&=\sum_{i\in I}\left\langle U^*(\Gamma^*_i\Gamma_i-\Lambda^*_i\Lambda_i)Tf,f\right\rangle+\sum_{i\in I}\left\langle U^*\Lambda^*_i\Lambda_iTf,f\right\rangle\\
		&\le R\|f\|^2+B\|f\|^2=(R+B)\|f\|^2.
	\end{align*}
	On the other hand
	\begin{align*}
		\sum_{i\in I}\left\langle U^*\Gamma^*_i\Gamma_iTf,f\right\rangle&=\sum_{i\in I}\left\langle U^*\Lambda^*_i\Lambda_iTf,f\right\rangle+\sum_{i\in I}\left\langle U^*(\Gamma^*_i\Gamma_i-\Lambda^*_i\Lambda_i)Tf,f\right\rangle\\
		&\ge \sum_{i\in I}\left\langle U^*\Lambda^*_i\Lambda_iTf,f\right\rangle-\sum_{i\in I}\left\langle U^*(\Gamma^*_i\Gamma_i-\Lambda^*_i\Lambda_i)Tf,f\right\rangle\\
		&\ge A\|f\|-R\|f\|^2=(A-R)\|f\|^2>0.
	\end{align*}
	So we have the result.  
\end{proof}

\begin{proposition}
	Let $T,U\in GL(\hs)$ and $\{\Lambda_i:i\in I\}$ be a $(T,U)$-controlled g-frame for $\hs$. Let $\{\Gamma_i:i\in I\}$ be a g-complete family of bounded operator. Suppose that $\Phi:\hs\longrightarrow\hs$ defined by
	$$\Phi(f)=\sum_{i\in I}U^*(\Gamma^*_i\Gamma_i-\Lambda^*_i\Lambda_i)Tf,~~\forall f\in\hs,$$
	is a positive and compact operator. Then $\{\Gamma_i:i\in I\}$ is a $(T,U)$-controlled g-frame for $\hs$.
\end{proposition}
\begin{proof}
	Let $\{\Lambda_i:i\in I\}$ be a $(T,U)$-controlled g-frame for $\hs$. By Proposition 2.3, then it is a g-frame for $\hs$ with bounds $A$ and $B$. On the other hand, since $\Phi$ is a positive compact operator, $U^{-1}\Phi T^{-1}$ is also a positive compact operator. Hence
	$$(U^*)^{-1}\Phi T^{-1}f=\sum_{i\in I}\Gamma^*_i\Gamma_if-\Lambda^*_i\Lambda_if,~~~\forall f\in\hs.$$	
	Let $\Psi=(U^*)^{-1}\Phi T^{-1}$ and $\Theta:\hs\longrightarrow\hs$ be an operator defined by 
	$$\Theta=S_{\Lambda}+\Psi.$$
	A simple computation shows that $\Psi$ is bounded and self-adjoint and $\Theta$ is bounded, linear, self-adjoint and 
	$$\Theta f=\sum_{i\in  I}\Gamma_i^*\Gamma_if,$$
	for any $f\in\hs$. Let $f$ be an arbitrary element of $\hs$, we have
	$$\|\Theta f\|=\|S_{\Lambda}f+\Psi f\|\le \|S_{\Lambda}f\|+\|\Psi f\|\le (B+\|\Psi\|)\|f\|.$$
	Therefore,
	$$\sum_{i\in I}\|\Gamma_if\|^2=\left\langle \Theta f,f\right\rangle \le (B+\|\Psi\|)\|f\|^2.$$
	Since $\Psi$ is a compact operator, $\Psi S^{-1}_{\Lambda}$ is also a compact operator on $\hs$. By Theorem 2.8 of \cite{asgari2005frames}, $\Theta$ has closed range. Now we show that $\Theta$ is injective. Let $g$ be an element of $\hs$ such that $\Theta g=0$, then 
	$$\sum_{i\in I}\|\Gamma_ig\|^2=\left\langle \Theta g,g\right\rangle =0.$$
	Hence $\Gamma_ig=0$ for each $i\in I$. Since $\{\Gamma_i\in L(\hs,\hs_i):i\in I\}$ is g-complete, we have $g=0$. Furthermore, we have
	$${\rm Range}(\Theta)=(N(\Theta^*))^{\perp}=N(\Theta)^{\perp}=\hs.$$
	Hence $\Theta$ is onto and therefore invertible on $\hs$. Similar to the proof of Theorem 2.8 of \cite{asgari2005frames}, we have 
	$$\sum_{i\in I}\|\Gamma_ig\|^2\ge(B+\|\Psi\|)^{-1}\|\Theta^{-1}\|^{-2}\|f\|^2.$$
	Then $\{\Gamma_i:i\in I\}$ is a g-frame for $\hs$.
	Since $\Phi=U^*S_{\Gamma}T-U^*S_{\Lambda}T$, $U^*S_{\Gamma}T=\Phi+U^*S_{\Lambda}T$. It is easy to see that $U^*S_{\Gamma}T$ is a bounded positive operator. Hence, we have $\{\Gamma_i:i\in I\}$  is a $(T,U)$-controlled g-frame for $\hs$. 	
\end{proof}	

Our next result is a generalization of Theorem 3.3 of \cite{casazza2008fusion}.
\begin{theorem}
	Let $T,U\in GL(\hs)$ and $\{\Lambda_i\in L(\hs,\hs_i):i\in I\}$ be a family of bounded operators. Let $\{\Gamma_{ij}\in L(\hs_i,\hs_{ij}):j\in J_i\}$ be a $(T,U)$-controlled g-frame for each $\hs_i$ with bounds $C_i$ and $D_i$, which $0<C\le C_i\le D_i\le D<\infty$. Then the following conditions are equivalent.
	\begin{enumerate}
		\item $\{\Lambda_i\in L(\hs,\hs_i):i\in I\}$ is a $(T,U)$-controlled g-frame for $\hs$.
		\item  $\{\Gamma_{ij}\Lambda_i\in L(\hs_i,\hs_{ij}):i\in I,~j\in J_i\}$ is a $(T,U)$-controlled g-frame for $\hs$.
	\end{enumerate}
\end{theorem}

\begin{proof}
	1$\rightarrow$2.	Let $\{\Lambda_i\in L(\hs,\hs_i):i\in I\}$ be a $(T,U)$-controlled g-frame with bounds ($A,B$) for $\hs$. Then for all $f\in\hs$ we have
	\begin{align*}
		\sum_{i\in I}\sum_{j\in J_i}\left\langle \Gamma_{ij}\Lambda_iTf,\Gamma_{ij}\Lambda_iUf\right\rangle&=\sum_{i\in I}\sum_{j\in J_i}\left\langle \Gamma^*_{ij}\Gamma_{ij}\Lambda_iTf,\Lambda_iUf\right\rangle\\
		&\le \sum_{i\in I}D_i\left\langle \Lambda_iTf,\Lambda_iUf\right\rangle\le DB\|f\|^2.
	\end{align*}
	Also we have 
	\begin{align*}
		\sum_{i\in I}\sum_{j\in J_i}\left\langle \Gamma_{ij}\Lambda_iTf,\Gamma_{ij}\Lambda_iUf\right\rangle&=\sum_{i\in I}\sum_{j\in J_i}\left\langle \Gamma^*_{ij}\Gamma_{ij}\Lambda_iTf,\Lambda_iUf\right\rangle\\
		&\ge \sum_{i\in I}C_i\left\langle \Lambda_iTf,\Lambda_iUf\right\rangle\ge CA\|f\|^2.
	\end{align*}
	
	2$\rightarrow$1. Let $\{\Gamma_{ij}\Lambda_i\in L(\hs_i,\hs_{ij}):i\in I,~j\in J_i\}$ be a $(T,U)$-controlled g-frame with bounds $A,B$ for $\hs$. Since $\Lambda_if\in\hs_i$, we have 
	$$\sum_{i\in I}\left\langle \Lambda_iTf,\Lambda_iUf\right\rangle\le \sum_{i\in I}\frac{1}{C_i}\sum_{j\in J_i}\left\langle \Gamma_{ij}\Lambda_iTf,\Gamma_{ij}\Lambda_iUf\right\rangle\le \frac{B}{C}\|f\|^2.$$
	Also
	$$\sum_{i\in I}\left\langle \Lambda_iTf,\Lambda_iUf\right\rangle\ge \sum_{i\in I}\frac{1}{D_i}\sum_{j\in J_i}\left\langle \Gamma_{ij}\Lambda_iTf,\Gamma_{ij}\Lambda_iUf\right\rangle\ge \frac{A}{D}\|f\|^2.$$ 
\end{proof}
The next result is a characterization for $(T,U)$-controlled g-frames.
\begin{theorem}
	Let $T,U\in GL(\hs)$ and $\{\Lambda_i\in L(\hs,\hs_i):i\in I\}$ be a family of bounded operators. Suppose that $\{e_{ij}:j\in J_i\}$ is an orthonormal basis for $\hs_i$ for each $i\in I$. Then $\{\Lambda_i:i\in I\}$ is a $(T,U)$-controlled g-frame for $\hs$ if and only if $\{T^*u_{ij}:i\in I,j\in J_i\}$ is a $U^*(T^*)^{-1}$-controlled frame for $\hs$, where $u_{ij}=\Lambda^*_ie_{ij}$.
\end{theorem}
\begin{proof}
	Let $\{e_{ij}:j\in J_i\}$ be an orthonormal basis for $\hs_i$ for each $i\in I$. For any $f\in\hs$, since $\Lambda_if\in\hs_i$, we have
	$$\Lambda_i(Tf)=\sum_{j\in J_i}\left\langle \Lambda_i(Tf),e_{ij}\right\rangle e_{ij}=\sum_{j\in J_i}\left\langle f,T^*\Lambda^*_ie_{ij}\right\rangle e_{ij}.$$
	Also,
	$$\Lambda_i(Uf)=\sum_{j\in J_i}\left\langle \Lambda_i(Uf),e_{ij}\right\rangle e_{ij}=\sum_{j\in J_i}\left\langle f,U^*\Lambda^*_ie_{ij}\right\rangle e_{ij}.$$
	Hence,
	$$\left\langle \Lambda_iTf,\Lambda_iUf\right\rangle =\sum_{j\in J_i}\left\langle f,T^*\Lambda^*_ie_{ij}\right\rangle \left\langle U^*\Lambda^*e_{ij},f\right\rangle .$$
	Now, if we take $u_{ij}=\Lambda^*_ie_{ij}$, $f_{ij}=T^*u_{ij}$ and $\Omega=U^*(T^*)^{-1}$, then 
	$$A\|f\|^2\le \sum_{i\in I}\left\langle \Lambda_iTf,\Lambda_iUf\right\rangle \le B\|f\|^2$$
	is equivalent to 
	$$A\|f\|\le\sum_{i\in I}\sum_{j\in J_i}\left\langle f,f_{ij}\right\rangle \left\langle \Omega f_{ij},f\right\rangle \le B\|f\|^2.$$
	So we have the result.		
\end{proof}

Note that $\{u_{ij}:i\in I,j\in J_i\}$ is the sequence induced by $\{\Lambda_i:i\in I\}$ with respect to $\{e_{ij}:j\in J_i\}$.

By the above result, finding suitable operator $T$ and $U$ such that $\{\Lambda_i:i\in I\}$ forms a $(T,U)$-controlled fusion frame for $\hs$ with optimal bounds, is equivalent to finding suitable operators $T$ and $U$ such that $\{T^*u_{ij}:i\in I,j\in J_i\}$ is a $U^*(T^*)^{-1}$-controlled frame for $\hs$ with optimal frame bounds.

Let $\hs$ and $\ks$ be two Hilbert spaces. We recall that $\hs\oplus\ks=\{(f,g):f\in\hs,g\in\ks\}$, is a Hilbert space with pointwise operations and inner product
$$\left\langle (f,g),(f',g')\right\rangle :=\left\langle f,f'\right\rangle _{\hs}+\left\langle g,g'\right\rangle _{\ks},~~\forall f,f'\in\hs,g,g'\in\ks.$$
Also if $\Lambda\in L(\hs, V)$ and $\Gamma\in L(\ks, W)$, then for all $f\in\hs, g\in\ks$ we define
$$\Lambda\oplus\Gamma\in L(\hs\oplus\ks,V\oplus W),~{\rm by}~(\Lambda\oplus\Gamma)(Tf,Ug):=(\Lambda Tf,\Gamma Ug),$$
where $V,W$ are Hilbert spaces and $T\in GL(\hs),~U\in GL(\ks)$.

\begin{theorem}
	Let $T\in GL(\hs),~U\in GL(\hs)$. Let $\{\Lambda_i\in L(\hs, V_i):i\in I\}$ and $\{\Gamma_i\in L(\ks, W_i):i\in I\}$ be $(T,T)$-controlled  g-frame with bounds $(A,B)$ and $(U,U)$-controlled g-frame with bounds $(C,D)$, respectively. Then $\{\Lambda_i\oplus\Gamma_i\in L(\hs \oplus\ks, V_i\oplus W_i):i\in I\}$ is a $(T,U)$-controlled g-frame with bounds $(\min\{A,C\},\max\{B,D\})$. 
\end{theorem}
\begin{proof}
	Let $(f,g)$ be an arbitrary element of $\hs\oplus\ks$. Then we have
	\begin{align*}
		\sum_{i\in I}\|(\Lambda_i\oplus\Gamma_i)(Tf,Ug)\|^2&=\sum_{i\in I}\left\langle(\Lambda_i\oplus\Gamma_i)(Tf,Ug),(\Lambda_i\oplus\Gamma_i)(Tf,Ug)\right\rangle  \\
		&= \sum_{i\in I}\left\langle (\Lambda_iTf,\Gamma_iUg),(\Lambda_iTf,\Gamma_iUg)\right\rangle \\
		&=\sum_{i\in I} \left\langle \Lambda_iTf,\Lambda_if\right\rangle +\left\langle \Gamma_iUg,\Gamma_iUg\right\rangle \\
		&=\sum_{i\in I}\|\Lambda_i Tf\|^2+\sum_{i\in I}\|\Gamma_iUf\|^2\\
		&\le B\|f\|^2+D\|g\|^2\\
		&\le\max\{B,D\}(\|f\|^2+\|g\|^2)=\max\{B,D\}\|(f,g)\|^2.
	\end{align*}
	Similarly we have
	$$\min\{A,C\}(\|f\|^2+\|g\|^2)\le 	\sum_{i\in I}\|(\Lambda_i\oplus\Gamma_i)(Tf,Ug)\|^2.$$
	So we have the result. 
\end{proof}
\section{Resolutions of the identity}
In this section, we will find new resolution of the identity. In fact, let $T,U\in GL(\hs)$ and $\{\Lambda_i\in L(\hs,\hs_i):i\in I\}$ be a $(T,U)$-controlled g-frame, then we have
$$f=\sum_{i\in I}S^{-1}_{T\Lambda U}U^*\Lambda^*_i\Lambda_iTf=\sum_{i\in I}U^*\Lambda^*_i\Lambda_iTS^{-1}_{T\Lambda U}f,~~\forall f\in\hs.$$
By choosing suitable controlled operators we may obtain more suitable approximations. Now we will give a new resolution of the identity by using two controlled operators.
\begin{definition}
	Let $T,U\in GL(\hs)$ and let $\{\Lambda_i\in L(\hs,\hs_i):i\in I\}$ and $\{\Gamma_i\in L(\hs,\hs_i):i\in I\}$ be $(T,T)$-controlled and $(U,U)$-controlled g-Bessel sequence, respectively. We define a $(T,U)$-controlled g-frame operator for this pair of controlled g-Bessel sequence as follows:
	$$S_{T\Gamma\Lambda U}(f)=\sum_{i\in I}U^*\Gamma^*_i\Lambda_iT(f),~~~\forall f\in\hs.$$
\end{definition}
As mentioned before,  $\{\Lambda_i\in L(\hs,\hs_i):i\in I\}$ and $\{\Gamma_i\in L(\hs,\hs_i):i\in I\}$ are also two g-Bessel sequence.
So by \cite{khosravi2008fusion} the g-frame operator $S_{\Gamma\Lambda }(f)=\sum_{i\in I}\Gamma^*_i\Lambda_i(f)$ for this pair of g-Bessel sequence is well defined and bounded. Since $S_{T\Gamma\Lambda U}=U^*S_{\Gamma\Lambda }T$, $S_{T\Gamma\Lambda U}$ is a well defined and bounded operator.
\begin{lemma}
	Let $T,U\in GL(\hs)$ and let $\{\Lambda_i:i\in I\}$ and $\{\Gamma_i:i\in I\}$ be $(T,T)$-controlled and $(U,U)$-controlled g-Bessel sequence with bounds $B_{T}$ and $B_{U}$, respectively. If $S_{T\Gamma\Lambda U}$ is bounded below, then $\{\Lambda_i:i\in I\}$ and $\{\Gamma_i:i\in I\}$ are $(T,T)$-controlled and $(U,U)$-controlled g-frames, respectively.
\end{lemma}
\begin{proof}
	Suppose that there exists a number $\lambda>0$ such that for all $f\in\hs$
	$$\lambda\|f\|\le\| S_{T\Gamma\Lambda U}\|,$$
	then we have
	\begin{align*}
		\lambda\|f\|\le\| S_{T\Gamma\Lambda U}\|&=\sup_{g\in\hs,\|g\|=1}\big|<\sum_{i\in I}U^*\Gamma^*_i\Lambda_iTf,g> \big|\\
		&=\sup_{\|g\|=1}\big|< \sum_{i\in I}\Lambda_iTf,\Gamma_iUg>\big |
		\\
		&\le \sup_{\|g\|=1}\big(\sum_{i\in I}\|\Lambda_iTf\|^2\big)^{1/2}\big(\sum_{i\in I}\|\Gamma_iUg\|^2\big)^{1/2}\\
		&\le \sqrt{B_{U}}\big(\sum_{i\in I}\|\Lambda_iTf\|^2\big)^{1/2}.
	\end{align*}	
	Hence,
	$$\frac{\lambda^2}{B_U}\|f\|^2\le \sum_{i\in I}\|\Lambda_iTf\|^2.$$
	On the other hand, since 
	$$S_{T\Gamma\Lambda U}^*=(U^*S_{\Gamma\Lambda}T)^*=T^*S_{\Gamma\Lambda}^*U=T^*S_{\Lambda\Gamma}U=S_{U\Lambda\Gamma T},$$
	we can say that $S_{U\Lambda\Gamma T}$ is also bounded below. So by the above result $\{\Gamma_i:i\in I\}$ is a $(U,U)$-controlled g-frame. 
\end{proof}

\begin{theorem}
	Let $T\in GL(\hs)$ and let $\Lambda=\{\Lambda_i\in L(\hs,\hs_i):i\in I\}$ be a $(T,T)$-controlled g-Bessel sequence. Then the following conditions are equivalent.
	\begin{enumerate}
		\item  $\Lambda$ is a $(T,T)$-controlled  g-frame for $\hs$.
		\item  There exists an operator $U\in GL(\hs)$ and a $(U,U)$-controlled g-Bessel sequence $\Gamma=\{\Gamma_i\in L(\hs,\hs_i): i\in I\}$ such that $S_{U\Gamma\Lambda T}\ge mI_{\hs}$ on $\hs$, for some $m>0$.
	\end{enumerate}
	
\end{theorem}
\begin{proof}1$\rightarrow$2. Let $\Lambda$ be a $(T,T)$-controlled g-frame with lower and upper g-frame bounds $A_{T}$ and $B_{T}$, respectively. Then we take $U=T$, $\Gamma_i=\Lambda_i$, for all $i\in I$. Hence we have 
	$$\left\langle S_{T\Lambda\Lambda T}f,f\right\rangle =\left\langle \sum_{i\in I}T^*\Lambda^*_i\Lambda_i Tf,f\right\rangle =\sum_{i\in I}\left\langle \Lambda_iTf,\Lambda_iTf\right\rangle \ge A_{T}\|f\|^2$$
	for all $f\in\hs$. Moreover,
	$$A_{T}\|f\|^2\le\|S_{T\Lambda\Lambda T}^{1/2}\|^2\le B_{T}\|f\|^2.$$

	By Lemma 2.2, $S_{T\Lambda\Lambda T}\in GL^+(\hs)$.
	
	2$\rightarrow$1. Suppose that there exists an operator $U\in GL(\hs)$ and a $(U,U)$-controlled g-Bessel sequence $\Gamma=\{\Gamma_i\in L(\hs,\hs_i):i\in I\}$ with Bessel bound $B_U$. Also let $m>0$ be a constant such that
	$$\left\langle S_{U\Gamma\Lambda T}f,f\right\rangle \ge m\|f\|^2$$
	for all $f\in\hs$. Then we have
	\begin{align*}
		m\|f\|^2&\le\left\langle S_{U\Gamma\Lambda T}f,f\right\rangle=\sum_{i\in I}\left\langle \Lambda_iTf,\Gamma_iUf\right\rangle \\
		&\le \big(\sum_{i\in I}\|\Lambda_iTf\|^2\big)^{1/2}\big(\sum_{i\in I}\|\Gamma_iUf\|^2\big)^{1/2}\\
		&\le \sqrt{B_U}\|f\|\big(\sum_{i\in I}\|\Lambda_iTf\|^2\big)^{1/2},
	\end{align*}
	by the Cauchy-Schwartz inequality. Hence,
	$$\frac{m^2}{B_U}\|f\|^2\le \sum_{i\in I}\|\Lambda_iTf\|^2\le B_T\|f\|^2.$$
	So $\Lambda$ is a $(T,T)$-controlled  g-frame for $\hs$. 
\end{proof}
\begin{theorem}
	Let $T,U\in GL(\hs)$ and let $\{\Lambda_i\in L(\hs,\hs_i):i\in I\}$ be a $(T,T)$-controlled g-frame with bounds $(A,B)$ for $\hs$. Let the family
	$\{\Gamma_i\in L(\hs,\hs_i):i\in I\}$ be a $(U,U)$-controlled g-Bessel sequence. Suppose that there exists a number $0\le \lambda\le A$ such that
	$$\|(S_{T\Gamma\Lambda U}-S_{T\Lambda T})f\|\le\lambda\|f\|,~~~\forall f\in\hs.$$
	Then $S_{T\Gamma\Lambda U}$ is invertible and also $\{\Gamma_i\in L(\hs,\hs_i):i\in I\}$ is a $(U,U)$-controlled g-frame for $\hs$.
\end{theorem}
\begin{proof}
	Let $f\in\hs$ be an arbitrary element of $\hs$, then we have
	\begin{align*}
		\|S_{T\Gamma\Lambda U}f\|&=\|S_{T\Gamma\Lambda U}f-S_{T\Lambda T}f+S_{T\Lambda T}f\| \\
		& \ge\|S_{T\Lambda T}f\|-\|S_{T\Gamma\Lambda U}f-S_{T\Lambda T}f\|\\
		&\ge(A-\lambda)\|f\|.
	\end{align*}
	So $S_{T\Gamma\Lambda U}$ is bounded below and therefore one-to-one with closed range. On the other hand, since
	$$\|S_{U\Gamma\Lambda T}-S_{T\Lambda T}\|=\|(S_{T\Gamma\Lambda U}-S_{T\Lambda T})^*\|\le\lambda,$$	
	by the above result $S_{U\Gamma\Lambda T}$ is also bounded below ($A-\lambda$) and therefore one-to-one with closed range. Hence both $S_{T\Gamma\Lambda U}$ and $S_{U\Gamma\Lambda T}$ are invertible.
	And,
	\begin{align}
		(A-\lambda)\|f\|\le\| S_{U\Gamma\Lambda T}\|&=\sup_{g\in\hs,\|g\|=1}\big|<\sum_{i\in I}T^*\Lambda^*_i\Gamma_iUf,g> \big|\nonumber\\
		&=\sup_{\|g\|=1}\big|< \sum_{i\in I}\Gamma_iUf,\Lambda_iTg>\big |
		\nonumber\\
		&\le \sup_{\|g\|=1}\big(\sum_i\|\Gamma_iUf\|^2\big)^{1/2}\big(\sum_i\|\Lambda_iTg\|^2\big)^{1/2}\nonumber\\
		&\le \sqrt{B}\big(\sum_i\|\Gamma_iUf\|^2\big)^{1/2}.\nonumber
	\end{align}	
	Hence,
	$$\frac{(A-\lambda)^2}{B}\|f\|^2\le \sum_{i\in I}\|\Gamma_iUf\|^2.$$ 
	Therefore, $\{\Gamma_i\in L(\hs,\hs_i):i\in I\}$ is a $(U,U)$-controlled g-frame for $\hs$.
\end{proof}
Another version of these cases is as follows.
\begin{proposition}
	Let $\Lambda$ and $\Gamma$ be controlled g-Bessel sequences as mentioned in Definition 3.1. Suppose that there exists $0<\varepsilon<1$ such that 
	$$\|f-S_{T\Gamma\Lambda U}f\|\le \varepsilon\|f\|,~~~\forall f\in\hs.$$
	Then $\Lambda$ and $\Gamma$ are $(T,T)$-controlled and $(U,U)$-controlled g-frames, respectively. Furthermore, $S_{T\Gamma\Lambda U}$ is invertible.
\end{proposition}
\begin{proof}
	Firstly 
	$$\|I_{\hs}-S_{T\Gamma\Lambda U}\|\le \varepsilon<1,$$
	therefore $S_{T\Gamma\Lambda U}$ is invertible.
	Secondly, let $f$ ban an arbitrary element of $\hs$. Then we have 
	$$\|S_{T\Gamma\Lambda U}f\|\ge \|f\|-\|f-S_{T\Gamma\Lambda U}f\|\ge(1-\varepsilon)\|f\|.$$
	Hence $S_{T\Gamma\Lambda U}$ is bounded below. By Lemma 3.2, we know that $\Lambda$ is a $(T,T)$-controlled g-frame. 
	
	On the other hand, we have 
	$$\|I_{\hs}-S_{U\Lambda\Gamma T}\|=\|(I_{\hs}-S_{T\Gamma\Lambda U})^*\|\le \varepsilon.$$
	Hence similarly we can say that $\Gamma$ is a $(U,U)$-controlled g-frame.  	
\end{proof}

With the hypotheses of Theorem 3.4 or Proposition 3.5, both $S_{T\Gamma\Lambda U}$ and $S_{U\Gamma\Lambda T}$ are invertible. Then the family 
$$\{S^{-1}_{T\Gamma\Lambda U}U^*\Gamma^*_i\Lambda_iT\}_{i\in I}$$
is a resolution of the identity. Also we have new reconstruction formulas as follows
$$f=\sum_{i\in I}S^{-1}_{T\Gamma\Lambda U}U^*\Gamma^*_i\Lambda_iTf=\sum_{i\in I}\Gamma^*_i\Lambda_iTS^{-1}_{T\Gamma\Lambda U}f$$
and 
$$f=\sum_{i\in I}S^{-1}_{U\Lambda\Gamma T}T^*\Lambda^*_i\Gamma_i Uf=\sum_{i\in I}T^*\Lambda^*_i\Gamma_iUS^{-1}_{U\Lambda\Gamma T}f.$$
Suppose that $\|I_{\hs}-S_{T\Gamma\Lambda U}\|<1$, then as we mentioned in Proposition 3.5, $S_{T\Gamma\Lambda U}$ is invertible and we have
$$S^{-1}_{T\Gamma\Lambda U}=\sum^{\infty}_{n=0}(I_{\hs}-S_{T\Gamma\Lambda U})^n.$$
Then we have
$$f=\sum_{i\in I}\sum^{\infty}_{n=0}(I_{\hs}-S_{T\Gamma\Lambda U})^nU^*\Gamma^*_i\Lambda_iTf=\sum_{i\in I}\sum^{\infty}_{n=0}U^*\Gamma^*_i\Lambda_iT(I_{\hs}-S_{T\Gamma\Lambda U})^nf.$$
Furthermore,
$$\|S^{-1}_{T\Gamma\Lambda U}\|\le (1-\|I_{\hs}-S_{T\Gamma\Lambda U}\|)^{-1}.$$
Therefore, 
$$\{(I_{\hs}-S_{T\Gamma\Lambda U})^nU^*\Gamma^*_i\Lambda_iT\}_{i\in I,n\in \z^+}$$
is a new resolution of the identity.
\section{Stability under perturbations}
Perturbation of frames is an important and useful objects to construct new frames from a given one. In this section we give new definitions of perturbations of g-frames  with respect to the operators $T,U$.
\begin{definition}
	Let  $T,U\in GL(\hs)$  and let $\{\Lambda_i\in L(\hs,\hs_i):i\in I\}$ and $\{\Gamma_i\in L(\hs,\hs_i):i\in I\}$ be two g-complete family of bounded operator. Let $0\le \lambda_1,\lambda_2<1$ be real numbers and let $\C=\{c_i\}_{i\in I}$ be an arbitrary sequence of positive numbers such that $\|\C\|_2<\infty$. We say that the family $\{\Gamma_i\in L(\hs,\hs_i):i\in I\}$ is a $(\lambda_1,\lambda_2,\C,T,U)$-perturbation of $\{\Lambda_i\in L(\hs,\hs_i):i\in I\}$ if we have 
	$$\|\Lambda_iTf-\Gamma_iUf\|\le\lambda_1\|\Lambda_iTf\|+\lambda_2\|\Gamma_iUf\|+c_i\|f\|,~~\forall f\in\hs.$$
\end{definition}

We have the following important result.
\begin{theorem}
	Let $\{\Lambda_i\in L(\hs,\hs_i):i\in I\}$ be a g-frame for $\hs$ with frame bounds $A,B$. Suppose that $T,U\in GL(\hs)$. Let $\{\Gamma_i\in L(\hs,\hs_i):i\in I\}$ be a $(\lambda_1,\lambda_2,\C,T,U)$-perturbation of $\{\Lambda_i\in L(\hs,\hs_i):i\in I\}$, in which 
	$$(1-\lambda_1)\sqrt{A}\|T^{-1}\|^{-1}>\|\C\|_2.$$
	Then $\{\Gamma_i\in L(\hs,\hs_i):i\in I\}$ is a g-frame for $\hs$ with g-frame bounds
	$$\left( \frac{(1-\lambda_1)\sqrt{A}\|T^{-1}\|^{-1}-\|\C\|_2}{1+\lambda_2}\|U\|^{-1}\right) ^2,\left(\frac{(1+\lambda_1)\sqrt{B}\|T\|+\|\C\|_2}{1-\lambda_2}\|U\|^{-1}\right) ^2$$
\end{theorem}
\begin{proof}
	Since $\{\Lambda_i\in L(\hs,\hs_i):i\in I\}$ be a g-frame for $\hs$ with frame bounds $A,B$, for all $f\in\hs$, we have 
	$$\frac{\sqrt{A}}{\|T^{-1}\|}\|f\|\le \sum_{i\in I}\left( \|\Lambda_iTf\|^2\right) ^{\frac{1}{2}}\le \sqrt{B}\|T\|f\|.$$
	Then by triangular inequality we have
	\begin{align*}
		\left( \sum_{i\in I}\|\Gamma_iUf\|^2\right) ^{\frac{1}{2}}
		&\le\left(\sum_{i\in I}(\|\Lambda_iTf\|+\|\Lambda_iTf-\Gamma_iUf\|)^2\right)^{\frac{1}{2}}  \\
		&\le\left(\sum_{i\in I}(\|\Lambda_iTf\|+\lambda_1\|\Lambda_iTf\|+\lambda_2\|\Gamma_iUf\|+c_i\|f\|)^2\right)^{\frac{1}{2}}\\
		&\le (1+\lambda_1)\sum_{i\in I}\left( \|\Lambda_iTf\|^2\right) ^{\frac{1}{2}}+\lambda_2\sum_{i\in I}\left( \|\Gamma_iUf\|^2\right) ^{\frac{1}{2}}+\|\C\|_2\|f\|.
	\end{align*}		
	Hence
	$$(1-\lambda_2)\sum_{i\in I}\left( \|\Gamma_iUf\|^2\right) ^{\frac{1}{2}}\le (1+\lambda_1)\sqrt{B}\|T\|\frac{\|Uf\|}{\|U\|^{-1}}+\|\C\|_2\frac{\|Uf\|}{\|U\|^{-1}}.$$
	Since $Uf\in\hs$, finally we have
	$$\sum_{i\in I}\|\Gamma_if\|^2\le \left( \frac{(1+\lambda_1)\sqrt{B}\|T\|+\|\C\|_2)}{1-\lambda_2}\|U\|^{-1}\right) ^2\|f\|^2.$$
	Now for the lower bound we have 
	\begin{align}
		\left( \sum_{i\in I}\|\Gamma_iUf\|^2\right) ^{\frac{1}{2}}
		&\ge\left(\sum_{i\in I}(\|\Lambda_iTf\|-\|\Lambda_iTf-\Gamma_iUf\|)^2\right)^{\frac{1}{2}}  \nonumber\\
		&\ge\left(\sum_{i\in I}(\|\Lambda_iTf\|-\lambda_1\|\Lambda_iTf\|-\lambda_2\|\Gamma_iUf\|-c_i\|f\|)^2\right)^{\frac{1}{2}}\nonumber\\
		&\ge (1-\lambda_1)\sum_{i\in I}\left( \|\Lambda_iTf\|^2\right) ^{\frac{1}{2}}-\lambda_2\sum_{i\in I}\left( \|\Gamma_iUf\|^2\right) ^{\frac{1}{2}}
		-\|\C\|_2\|f\|.\nonumber
	\end{align}
	Hence 
	$$(1+\lambda_2)	\sum_{i\in I}\left( \|\Gamma_iUf\|^2\right) ^{\frac{1}{2}}\ge(1-\lambda_1)\sqrt{A}\|T^{-1}\|^{-1}\frac{\|Uf\|}{\|U\|^{-1}}-\|\C\|_2\frac{\|Uf\|}{\|U\|^{-1}},$$ 
	which yields
	$$\sum_{i\in I}\|\Gamma_if\|^2\ge\left( \frac{(1-\lambda_1)\sqrt{A}\|T^{-1}\|^{-1}-\|\C\|_2}{1+\lambda_2}\|U\|^{-1}\right) ^2\|f\|^2.$$
	Therefore, we get the results.	
\end{proof}
\section*{Acknowledgements}
The research is supported by the National Natural Science Foundation of China (Nos. 11271001 and 61370147).
\bibliographystyle{amsplain}

\bibliography{mybib} 

\providecommand{\bysame}{\leavevmode\hbox to3em{\hrulefill}\thinspace}
\providecommand{\MR}{\relax\ifhmode\unskip\space\fi MR }
\providecommand{\MRhref}[2]{%
  \href{http://www.ams.org/mathscinet-getitem?mr=#1}{#2}
}
\providecommand{\href}[2]{#2}
\begin{thebibliography}{10}

\bibitem{asgari2005frames}
MS~Asgari and Amir Khosravi, \emph{Frames and bases of subspaces in Hilbert
  spaces}, Journal of mathematical analysis and applications \textbf{308}
  (2005), no.~2, 541--553.

\bibitem{balazs2010weighted}
Peter Balazs, Jean-Pierre Antoine, and ANNA GRYBO{\'S}, \emph{Weighted and
  controlled frames: Mutual relationship and first numerical properties},
  International journal of wavelets, multiresolution and information processing
  \textbf{8} (2010), no.~01, 109--132.

\bibitem{balazs2012multipliers}
Peter Balazs, Dominik Bayer, and Asghar Rahimi, \emph{Multipliers for
  continuous frames in Hilbert spaces}, Journal of Physics A: Mathematical and
  Theoretical \textbf{45} (2012), no.~24, 244023.

\bibitem{bodmann2005frames}
Bernhard~G Bodmann and Vern~I Paulsen, \emph{Frames, graphs and erasures},
  Linear Algebra and its Applications \textbf{404} (2005), 118--146.

\bibitem{bogdanova2005stereographic}
Iva Bogdanova, Pierre Vandergheynst, Jean-Pierre Antoine, Laurent Jacques, and
  Marcella Morvidone, \emph{Stereographic wavelet frames on the sphere},
  Applied and Computational Harmonic Analysis \textbf{19} (2005), no.~2,
  223--252.

\bibitem{casazza2008fusion}
Peter~G Casazza, Gitta Kutyniok, and Shidong Li, \emph{Fusion frames and
  distributed processing}, Applied and computational harmonic analysis
  \textbf{25} (2008), no.~1, 114--132.

\bibitem{Daubechies1986}
Ingrid Daubechies, A.~Grossmann, and Y.~Meyer, \emph{Painless nonorthogonal
  expansions}, Journal of Mathematical Physics \textbf{27} (1986), no.~5.

\bibitem{duffin1952class}
Richard~J Duffin and Albert~C Schaeffer, \emph{A class of nonharmonic Fourier
  series}, Transactions of the American Mathematical Society \textbf{72}
  (1952), no.~2, 341--366.

\bibitem{guo2014characterizations}
Xunxiang Guo, \emph{Characterizations of disjointness of g-frames and
  constructions of g-frames in Hilbert spaces}, Complex Analysis and Operator
  Theory \textbf{8} (2014), no.~7, 1547--1563.

\bibitem{khosravi2008fusion}
Amir Khosravi and Kamran Musazadeh, \emph{Fusion frames and g-frames}, Journal
  of Mathematical Analysis and Applications \textbf{342} (2008), no.~2,
  1068--1083.

\bibitem{kovavcevic2008introduction}
Jelena Kova{\v{c}}evi{\'c} and Amina Chebira, \emph{An introduction to frames},
  Foundations and Trends in Signal Processing \textbf{2} (2008), no.~1, 1--94.

\bibitem{leng2011optimal}
Jinsong Leng and Deguang Han, \emph{Optimal dual frames for erasures II},
  Linear Algebra and its Applications \textbf{435} (2011), no.~6, 1464--1472.

\bibitem{leng2011optimald}
Jinsong Leng, Deguang Han, and Tingzhu Huang, \emph{Optimal dual frames for
  communication coding with probabilistic erasures}, Signal Processing, IEEE
  Transactions on \textbf{59} (2011), no.~11, 5380--5389.

\bibitem{morillas2017optimal}
Patricia~Mariela Morillas, \emph{Optimal dual fusion frames for probabilistic
  erasures}, Electronic Journal of Linear Algebra \textbf{32} (2017), 191--203.

\bibitem{musazadeh2016some}
Kamran Musazadeh and Hassan Khandani, \emph{Some results on controlled frames
  in Hilbert spaces}, Acta Mathematica Scientia \textbf{36} (2016), no.~3,
  655--665.

\bibitem{rahimi2013controlled}
Asghar Rahimi and Abolhassan Fereydooni, \emph{Controlled g-frames and their
  g-multipliers in Hilbert spaces}, Analele Universitatii" Ovidius"
  Constanta-Seria Matematica \textbf{21} (2013), no.~2, 223--236.

\bibitem{Sun2015Nonlinear}
Qiyu Sun and Wai~Shing Tang, \emph{Nonlinear frames and sparse reconstructions
  in Banach spaces}, Journal of Fourier Analysis and Applications (2015),
  1--35.

\bibitem{sun2006g}
Wenchang Sun, \emph{G-frames and g-Riesz bases}, Journal of Mathematical
  Analysis and Applications \textbf{322} (2006), no.~1, 437--452.

\bibitem{sun2007stability}
\bysame, \emph{Stability of g-frames}, Journal of mathematical analysis and
  applications \textbf{326} (2007), no.~2, 858--868.

\end{thebibliography}

\end{document}